\documentclass{amsart}

\usepackage{amsmath,empheq, amsthm}
\usepackage{amssymb}

\newcommand{\R}{{\mathbb{R}}}
\newcommand{\N}{{\mathbb{N}}}

\newtheorem{definition}{Definition}
\newtheorem{theorem}{Theorem}

\newtheorem{lemma}{Lemma}

\newtheorem{proposition}{Proposition}
\newtheorem{remark}{Remark}

\begin{document}

\title[Linearization of a nonautonomous unbounded system]{Linearization of a nonautonomous unbounded system with nonuniform contraction: A Spectral Approach}
\author{ I. Huerta }
\address[I. Huerta]{Departamento de Matem\'atica y C.C, Facultad de Ciencia, Universidad de Santiago de Chile
Casilla 307, Correo 2, Santiago, Chile}
\email{ignacio.huerta@usach.cl}
\thanks{This research has been partially supported  FONDECYT Regular 1170968 and CONICYT-PCHA/2015-21150270}
\date{\today}

\maketitle

\begin{abstract}
For a nonautonomous linear system with nonuniform contraction, we construct a topological equivalence between this system and an unbounded nonlinear perturbation. This topological equivalence is constructed as a composition of homeomorphisms. The first one is set up by considering the fact that linear system is almost reducible to diagonal system with a small enough perturbation where the diagonal entries belong to spectrum of the nonuniform exponential dichotomy; and the second one is constructed in terms of the crossing times with respect to unit sphere of an adequate Lyapunov function associated to the linear system.
\end{abstract}

\section{Introduction}

The well known Hartman--Grobman's Theorem is an essential tool for the study of the local behavior of autonomous and nonautonomous nonlinear dynamical systems. This theorem establishes the existence of a local topological conjugacy  between the solutions of a nonlinear system with its linearization around an hy\-per\-bo\-lic e\-qui\-li\-brium, \textit{i.e.}, the dynamics are topologically the same in a neighborhood of the equilibrium point. 
The global behavior study begins in 1969, when C. Pugh \cite{Pugh} studied a particular case of the Hartman--Grobman's Theorem focused to linear systems with bounded and Lipschitz perturbations allowing the construction of an explicit and global homeomorphism.

\subsection{Nonautonomous Linearization and Unbounded Systems}

The work cited above inspired K.J. Palmer \cite{Palmer} to achieve the first result of global linearization in the nonautonomous framework. The seminal  article of K.J. Palmer and its extensions \cite{J,Shi} consider vector fields whose linear component inherits, in some sense, the hyperbolicity property of the autonomous case; while the nonlinear part satisfies boundedness and Lipschitzness properties. 

A remarkable extension of the previous work was made F. Lin \cite{Lin2}, who considered this problem by dropping the boundedness of the nonlinear perturbations, opening new ideas and methods. The work of Lin is mainly based in three steps:

\noindent (i) The linear system 
\begin{equation}
\label{lin1intro}
\dot{x} = A(t)x
\end{equation}
is supposed to be uniformly asymptotically stable, then it can be reduced to the linear system
\begin{equation}
\label{lin2intro}
\dot{x} = [C(t) + B(t)]x
\end{equation}
where $C(t)$ is diagonal, $B(t)$ is bounded, and the diagonal part is contained in the spectrum associated to nonautonomous hyperbolicity, a formal definition will be given later.

\noindent (ii)  The system
\begin{equation}
\label{nolin2intro}
\dot{x} = [C(t) + B(t)]x + g(t,x)
\end{equation}
is topologically equivalent to an autonomous linear system which is uniformly asymptotically stable,
where $g(t,x)$ has an equilibrium point at the origin for any $t \in \mathbb{R}$ and is Lipschitz with constant dependent of the boundedness of $B(t)$. The construction
of this topological equivalence is made by using the concept of \emph{crossing times} with respect to unit sphere. Notice that a suitable Lyapunov function is used to find this crossing times.

\noindent (iii) A chain of homeomorphisms which are involved with the topological equivalences constructed in steps (i) and (ii). 

%\subsection{Nonuniform contractibility}
\subsection{Structure and novelty of the article}
The section 2 states and comments the properties of the linear system (\ref{lin1intro}) and the nonlinear perturbations which will be con\-si\-dered in this work. Additionally, 
we recall the main tools that will be used along this article, namely, the property of topological equivalence, the nonuniform exponential dichotomy and its associated spectrum, the $\delta$--nonuniform kinematical similarity and
the nonuniform almost reducibility.

The section 3 is devoted to characterizing the property of nonuniform contraction in terms of an adequate Lyapunov function and positive quadratic forms. We give appropriate definitions of these concepts in a nonuniform context. Moreover, we state the main Theorems of this work; \textbf{(Ia)}: we will show that if the linear system (\ref{lin2intro}) has a nonuniform contraction then is topologically equivalent to a system (\ref{nolin2intro}) whose nonlinearity satisfy suitable properties. \textbf{(Ib)}: Moreover, we will prove that the Lyapunov function associated to  (\ref{lin2intro})  has a relation with the behavior of the solutions of the perturbed system (\ref{nolin2intro}).  \textbf{(II)}: We will prove that if the linear system  (\ref{lin1intro}) satisfies subtle conditions then is topologically equivalent to 
the perturbed system 
$$
\dot{x}=A(t)x+f(t,x)
$$
where $t\mapsto f(t,0)=0$ and $x\mapsto f(t,x)$ is Lipschitz for any $t\in \mathbb{R}_{0}^{+}$. \textbf{(III)}: We will generalize the above result considering the boundedness of $t\mapsto f(t,0)$ instead of
the vanishing at the origin. 

In the section 4 we generalize to the nonuniform context a classical result of local continuity with respect to
the initial conditions.

The last sections are devoted to the proofs of our results. In the section 5 we will follow the lines of the proof of Palmer's Lemma in \cite{PalmerEqTop} in order to prove (Ib). Note that the result of Palmer is immersed in a uniform context, while ours is in a nonuniform framework which entail technical difficulties. To obtain (Ia) we will use the  \emph{crossing times} defined by the Lyapunov functions associated to contractive linear system. Moreover, at this point, we establish the  major difference with the Lin's work: Now, it is impossible carry out the step (ii) described in Section 1.1, \textit{i.e.}, we can not construct a topological equivalence between a nonlinear system with nonuniform contractive linear part and an autonomous linear system which is uniformly stable  by intrinsic nature. In the section 6 is done the proof of (II) which is immediately consequence of the (Ib). Finally, the section 7 is devoted to proof of (III).

We emphasize that only few results of topological equivalence consider the unbounded nonlinearity. To the best of our knowledge, this property of unboundedness is only considered in \cite{Lin2}, \cite{CR2018} (in a differential and discrete uniform context, respectively) and \cite{Xia} in an impulsive framework.

\section{Preliminaries}
\subsection{Properties}
In this article we consider the following couple of the systems

\begin{subequations}
  \begin{empheq}[left=\empheqlbrace]{align}
    & \dot{x} = A(t)x \label{lin1} \\
    & \dot{x} = A(t)x + f(t,x) \label{nolin1}
  \end{empheq}
\end{subequations}
where $A \colon \mathbb{R}_0^+ \to M(n, \mathbb{R})$  and $f: \mathbb{R}_{0}^{+} \times \mathbb{R}^n \to \mathbb{R}^n$ is continuous on $(t,x)$ and

\begin{subequations}
  \begin{empheq}[left=\empheqlbrace]{align}
    & \dot{y} =[C(t)+B(t)]y \label{lin2} \\
    & \dot{y} = [C(t)+B(t)]y + g(t,y) \label{nolin2}
  \end{empheq}
\end{subequations}
where $B, C \colon \mathbb{R}_0^+ \to M(n, \mathbb{R})$  and $g: \mathbb{R}_{0}^{+} \times \mathbb{R}^n \to \mathbb{R}^n$ is continuous on $(t,x).$ Additionally, the following properties are verified:

\begin{itemize}
    \item [\textbf{(P1)}] For $\nu, \mathcal{M} > 0,$  $\left \|A(t)\right \|\leq \mathcal{M} \exp(\nu t) $  for any $t \in \mathbb{R}_{0}^{+}.$
    \vspace{0.2 cm}
    \item [\textbf{(P2)}] The evolution operator $\Phi (t,s)$ of  (\ref{lin1}) has a nonuniformly bounded growth (\cite{Zhang}), namely, there exist constants $K_0\geq1$, $a\geq0$ and $\bar{\varepsilon}\geq0$ such that $$\left \| \Phi (t,s) \right \| \leq K_0\exp(a\left | t-s \right | + \bar{\varepsilon}s), \quad t,\ s \in\R_0^{+}.$$
    \item [\textbf{(P3)}] The system (\ref{lin1}) is  nonuniform contractible if there exist $K>0$, $\alpha>0$ and $\mu\geq0$ such that
\begin{equation}
\label{estabilidadexponencialnouniforme}
     \left \| \Phi(t,s)\right \|\leq K\exp(-\alpha (t-s)+\mu s)\quad \text{for any $t\geq s\geq 0$.}
    %$\text{for any t\geq s}$.
\end{equation}
    \item [\textbf{(P4)}] The function $f$ is continuous on $(t,x)$ and is an element of one of the following families of functions:
    \begin{equation*}
    \mathcal{A}_1=\left \{
    \begin{array}{c}
     f: \sup_{t\in\R_0^{+}}\left \| f(t,0)\right \|<+\infty \ \textnormal{and} \ \exists \ L_f, \ \beta\geq0 \ \textnormal{s.t.} \\
     \left \| f(t,u)-f(t,v)\right \|\leq L_f\exp(-2\beta t)\left \|u-v\right \|\ \forall t\in\R_0^{+}
    \end{array}
    \right \},
    \end{equation*}
    \medskip
    \begin{equation*}
        \mathcal{A}_2=\left \{ f: f\in\mathcal{A}_1 \ \textnormal{and} \ f(t,0)=0 \ \textnormal{for all} \ t\in\R_0^{+}\right \}.
    \end{equation*}
\end{itemize}

\subsection{Main Tools}

The fundamental tools in our work are the concepts of topological equivalence, introduced by K.J. Palmer in \cite{Palmer}, the nonuniform exponential dichotomy which was introduced  by L. Barreira and C. Valls in \cite{BV-CMP}, and the $\delta$-nonuniform kinematical similarity.

\begin{definition}
\label{eqtopnouniforme}
The systems \textnormal{(\ref{lin1})} and \textnormal{(\ref{nolin1})} will be called topologically equivalent if there exists a map $H:\R_0^{+}\times\R^n \to \mathbb{R}^{n}$ with the properties
\begin{itemize}
    \item [\textnormal{(i)}] For each fixed $t\in\R_0^{+}$, the map $\xi\mapsto H(t,\xi)$ is a bijection.
    \item [\textnormal{(ii)}] For any fixed $t\in\R_0^{+}$, the maps $\xi\mapsto H(t,\xi)$ and $\xi\mapsto H^{-1}(t,\xi)=G(t,\xi)$ are continuous.
    \item [\textnormal{(iii)}] If $\left \|\xi\right \|\rightarrow+\infty$, then $\left \|H(t,\xi)\right \|\rightarrow+\infty$.
    \item [\textnormal{(iv)}] If $x(t)$ is a solution of \textnormal{(\ref{lin1})}, then $H(t,x(t))$ is a solution of \textnormal{(\ref{lin2})}. Similarly, if $y(t)$ is a solution of \textnormal{(\ref{nolin1})}, then $G(t,y(t))$ is a solution of \textnormal{(\ref{lin1})}.
\end{itemize}
\end{definition}

\begin{definition}\textnormal{(\cite{BV-CMP}, \cite{Chu}, \cite{Zhang})}
\label{NUED}
The system \textnormal{(\ref{lin1})} has a nonuniform exponential dichotomy on $\R_0^+$ if there exist an invariant projector $P(\cdot)$, constants $K\geq 1$, $\alpha>0$ and $\mu\geq 0$ such that
\begin{equation}
\left\{\begin{array}{rcl}
%\label{nouniforme}
\left \| \Phi(t,s)P(s) \right \|&\leq& K\exp(-\alpha(t-s)+\mu\left | s \right |),\quad t\geq s, \quad t, s \in \R_0^+,                      \\
\left \| \Phi(t,s)(I-P(s)) \right \|&\leq& K\exp(\alpha(t-s)+\mu\left | s \right |),\quad t\leq s, \quad t, s \in \R_0^+.
\end{array}\right.
\end{equation}
\end{definition}

\begin{definition}\textnormal{(\cite{Zhang})}
Given a $\delta>0$, the linear system \textnormal{(\ref{lin1})} is $\delta$-nonuniformly kinematically similar to 
\begin{equation}
\label{similaridadcinematicanouniforme}
\dot{y}=U(t)y,
\end{equation}
if there exist a Lyapunov's transformation $S(\delta,t)$ and $\upsilon\geq0$, with $$\left \|S(\delta, t)\right \|\leq M_{\upsilon, \delta}\exp(\upsilon t) \quad and \quad \left \|S^{-1}(\delta, t)\right \|\leq M_{\upsilon, \delta}\exp(\upsilon t),$$ 
such that the change of coordinates $y(t)=S^{-1}(\delta,t)x(t)$ transforms the system (\ref{lin1}) into (\ref{similaridadcinematicanouniforme}).
\end{definition}

\begin{remark}
\label{obs1}
The nonuniform kinematical similarity preserves the nonuniform contraction (see more details in \cite{Cast-Huerta}, Lemma 2). Thus, as the systems (\textnormal{\ref{lin1})} and \textnormal{(\ref{lin2})} are $\delta$--nonuniformly kinematically similar (see Theorem 1 of \cite{Cast-Huerta}), and as the system \textnormal{(\ref{lin1})} satisfies the condition \textnormal{(\textbf{P2})} with $K\geq 1$, $\alpha>0$, $\mu\geq 0$ and if $\alpha>\mu$, then the system \textnormal{(\ref{lin2})} admits a nonuniform contraction, i.e., there exist $K_1\geq 1$, $\alpha_1>0$ and $\mu_1\geq 0$ 
%such that $\alpha_1>\mu_1$, 
satisfying 

\begin{equation}
\left \| \Psi(t,s)\right \|\leq K_1 \exp(-\alpha_1 (t-s)+\mu_1 s), \ \ t\geq s \geq 0,
\end{equation}
 where $\Psi(t,s)$ is the evolution operator of \textnormal{(\ref{lin2})}.
\end{remark}

\begin{remark}
It is easy to verify that the property of $\delta$-nonuniform kinematical similarity is a equivalence relation and a particular case of nonuniform topological equi\-valence. Indeed, the properties of Definition \ref{eqtopnouniforme} are verified with $H(t,\xi)=S^{-1}(\delta,t)\xi.$
\end{remark}

Now we recall
definition of \textit{nonuniform almost reducibility} which is a generalization of the concept of almost reducibility introduced by B.F. Bylov \cite{Bylov} in a uniform context.

\begin{definition}\textnormal{(\cite{Cast-Huerta})}
The system \textnormal{(\ref{lin1})} is nonuniformly almost reducible to
$$\dot{y}=C(t)y,$$
if for any $\delta>0$ and $\varepsilon \geq0$, there exists a constant $K_{\delta,\varepsilon}\geq1$ such that \textnormal{(\ref{lin1})} is $\delta-$nonuniformly kinematically similar to
$$\dot{y}=[C(t)+B(t)]y,\quad with \quad \left \|B(t)\right \|\leq \delta K_{\delta,\varepsilon}$$
for any $t\in\R_0^{+}$.
\end{definition}

The $ \varepsilon $ parameter represents the nonuniform context in which we are working. In the case when $C(t)$ is a diagonal matrix, if $K_{\delta,\varepsilon}=1$ it is said that \textnormal{(\ref{lin1})} is almost reducible to a diagonal system and it was proved in \cite{Bylov} that any continuous linear system satisfies this property and the components of $C(t)$ are real numbers. This notion of almost reducibility to a diagonal system was rediscovered and
improved by F. Lin in \cite{Lin}, who introduces the concept of contractibility. Lin showed that the Sacker and Sell spectrum is the minimal compact set where  $C(t)$ is contained. This minimal set which $C(t)$ is contained is called \emph{contractible set}. In a nonuniform context, in \cite{Cast-Huerta} we proved that $C(t)$ is contained in the spectrum associated to nonuniform exponential dichotomy.

\begin{definition}\textnormal{(\cite{Chu} ,\cite{Zhang})}
The nonuniform spectrum (also called nonuniform exponential dichotomy spectrum) of (\ref{lin1}) is the set $\Sigma(A)$ of $\lambda\in\R$ such that the systems
 \begin{equation}
 \label{sistemaperturbado}
\dot{x}=[A(t)-\lambda I]x
 \end{equation}
 have not nonuniform exponential dichotomy on $\R_0^{+}.$
\end{definition}

\begin{remark}
Assumptions  \textnormal{{\bf{(P2)}}} and \textnormal{{\bf{(P3)}}} have a strong relation with the set $\Sigma(A).$ Indeed,
{\bf{(P2)}} implies that $\Sigma(A)$ is a finite union of at most $m\leq n$ compact intervals 
$$\Sigma(A)=\bigcup_{i=1}^{m}[ a_i,b_i ],$$
with $-\infty<a_1\leq b_1<\ldots<a_m\leq b_m <+\infty$ (see \cite{Chu}, \cite{Xia2019},\cite{Zhang}).
%{\bf{(P1)}} and {\bf{(P2)}} implies that (\ref{ecuacionlineal}) is $\delta$-nonuniformly kinematically similar to (\ref{ecuacioncontractibilidad}). 
On the other hand, {\bf{(P3)}} implies that $\Sigma(A)\subset (-\infty,0)$.
\end{remark}

\begin{remark}
In \cite[Theorem 1]{Cast-Huerta} it was proved that if {\bf{(P1)}} and {\bf{(P2)}} are satisfied, the system \textnormal{(\ref{lin1})} is $\delta$-nonuniformly kinematically similar via $S^{-1}(\delta,t)$ to \textnormal{(\ref{lin2})}, then there exist constants $M_1>0$ and $\beta\geq0$ such that 
$$\left \| S(\delta,t)\right \|\leq M_1\exp(\beta t)\quad and\quad \left \| S^{-1}(\delta,t)\right \|\leq M_1\exp(\beta t).$$ 

Moreover, $C(t)=\textnormal{Diag}(C_1(t),\dots ,C_n(t))$ with $C_i(t)\in\Sigma(A)$ and $\left \| B(t)\right \|\leq \delta K_{\delta,\varepsilon}$.

In addition, under the same transformation, the system \textnormal{(\ref{nolin1})} is transformed in 
\begin{equation}
    \label{sistematransformado}
    \dot{y}=(C(t)+B(t))y+S^{-1}(\delta,t)f(t,S(\delta,t)y).
\end{equation}
\end{remark}

\section{Main Results}

\subsection{Nonuniform Contractions and Lyapunov Functions}
In this section, for the system (\ref{lin1}), we obtain a complete characterization of nonuniform contraction  in terms of a Lyapunov function which will allow us to construct a topological equivalence between systems \textnormal{(\ref{lin1})}--\textnormal{(\ref{nolin1})} and \textnormal{(\ref{lin2})}--\textnormal{(\ref{nolin2})}. For this purpose, we recall the definition of \emph{strict Lypaunov function} and the main results from \cite{Liao}.  

\begin{definition}
Given $\mathcal{K}\geq 1$ and $\upsilon \geq 0.$ We say that a continuous function \newline$V:[0,+\infty)\times X\rightarrow\R_0^{+}$, where $X$ is a Banach space, is a strict Lyapunov function for \textnormal{(\ref{lin1})} if

\begin{itemize}
\item[\textnormal{\textbf{(V1)}}] $\left \|x\right \|^2\leq V(t,x)\leq \mathcal{K}^2\exp(2\upsilon t) \left \|x\right \|^2$, for any $t\geq 0$ and $x\in X$,
\\
\item[\textnormal{\textbf{(V2)}}] $V(t,\Phi(t,s)x)\leq V(s,x)$, for any $t\geq s\geq 0$ and $x\in X$,
\\
\item[\textnormal{\textbf{(V3)}}] Exists $\gamma>0$ such that $V(t,\Phi(t,s)x)\leq\exp(-2\gamma(t-s))V(s,x)$, $\forall t\geq s\geq 0$ and $x\in X$.
\end{itemize}
\end{definition}

It is fair to say that unlike the classical definitions of the concept of Lyapunov functions, in the previous definition it is not requested that $V$ be differentiable.

Additionally, the last definition has subtle differences with respect to Liao  {\it{et al.}} \cite{Liao}. In fact, we have tailored it in order to relate it with the nonuniform exponential dichotomy and it is for that reason that the proof of the following result is inspired by the ideas presented by Liao \textit{et al.} in their work. Indeed, we have the following result.

\begin{theorem}
The system \textnormal{(\ref{lin1})} has nonuniform contraction if and only if it admits a strict Lyapunov function.
\end{theorem}

\begin{proof}

Suppose that there exists a strict Lyapunov function for (\ref{lin1}). From the conditions \textbf{(V1)} and \textbf{(V3)} we have 
%\begin{equation} 
$$
\begin{array}{rl}
\left \|\Phi(t,s)x\right \|^2 \leq & V(t,\Phi(t,s)x)\leq \  \exp(-2\gamma(t-s))V(s,x),  \\
\leq & \exp(-2\gamma(t-s))\mathcal{K}^2\exp(2\upsilon s)\left \|x\right \|^2 
\end{array}
$$ 
which implies that 
$$\left \|\Phi(t,s)x\right \|\leq \mathcal{K}\exp(-\gamma(t-s)+\upsilon s)\left \|x\right \|.$$
Therefore, (\ref{lin1}) admits a nonuniform contraction with $\gamma = \alpha$ and $\upsilon = \mu.$

On the other hand, for $t\geq0$ and $x\in X$ we define
$$V(t,x)=\sup _{\tau\geq t}\left \{\left \|\Phi(\tau,s)x\right \|^2\exp(2\alpha(\tau - t))\right \}.$$
As (\ref{lin1}) admits nonuniform contraction, we have that $V(t,x)\leq K^2\exp(2\mu t)\left \|x\right \|^2$. If we consider $\tau=t$, then $\left \|x\right \|^2\leq V(t,x)$. Now, for $t\geq s\geq0$
$$
\begin{array}{rcl}
     V(t,\Phi(t,s)x) & = & \sup_{\tau\geq t}\left \{ \left \|\Phi(\tau,t)\Phi(t,s)x\right \|^{2}\exp(2\alpha (\tau - t))\right \}, \\\\
     & = & \exp(2\alpha (s-t)) \sup_{\tau\geq t}\left \{ \left \|\Phi(\tau,s)x\right \|^{2}\exp(2\alpha (\tau - s))\right \}, \\\\
     & \leq & \exp(2\alpha (s-t)) \sup_{\tau\geq s}\left \{ \left \|\Phi(\tau,s)x\right \|^{2}\exp(2\alpha (\tau - s))\right \}, \\\\
     & = & \exp(-2\alpha (t-s))V(s,x).
\end{array}
$$
Therefore, $V$ is a strict Lyapunov function for (\ref{lin1}).

\end{proof}

Now we will focus in  Lyapunov functions that are defined in terms of quadratic forms.
Let $\mathcal{S}(t)\in\mathcal{B}(X)$ be a symmetric positive-definite operator for $t\geq 0,$ where $\mathcal{B}(X)$  the space of bounded linear
operators in a Banach space $X$; in our case $X=\R^n$. Given the operator $\mathcal{S}$ as before, a $\mathit{quadratic\  Lyapunov\  function }$ V is given as follows
\begin{equation}
\label{formacuadratica}
V(t,x)=\left <\mathcal{S}(t)x,x \right >,
\end{equation} 
where $\left <\cdot,\cdot\right >$ is the usual inner product of $\R^n$.

\begin{remark}
Given two linear operators $M,N$, we write $M \leq N$ if they verify $\left < Mx,x\right >\leq\left < Nx,x\right >$ for $x\in X$.
\end{remark}

The following result (see \cite[Theorem 2.2]{Liao} with $\mu(t) = e^t$ ) establishes a characte\-rization of nonuniform contraction in terms of the existence of quadratic Lypaunov function. 

\begin{proposition}
\label{chinoshindawi}
We assume that there exist constants $c>0$ and $d\geq 1$ such that 
\begin{equation}
\label{cotaoperadordeevolucion}
    \left \| \Phi(t,s)\right \|\leq c, \ \ whenever \ \ t-s\leq \ln(d) 
\end{equation}
Then \textnormal{ (\ref{lin1})} admits a nouniform contraction if and only if there exist symmetric positive definite operators $\mathcal{S}(t)$ and constant $\mathcal{C}$,$\mathcal{K}_1>0$ such that $S(t)$ is of class $C^{1}$ in $t\geq0$ and 
\begin{equation}
\label{cotaparaS}
    \left \| \mathcal{S}(t)\right \|\leq \mathcal{C} \mathcal{K}_1 \exp(2\mu t),
\end{equation}
\begin{equation}
\label{cotanegativa}
    \mathcal{S}'(t)+A^{*}(t)\mathcal{S}(t)+\mathcal{S}(t)A(t)\leq (-Id+\mathcal{K}_1\mathcal{S}(t)).
\end{equation}
\end{proposition}

\begin{remark}
In the proof of this result, the condition (\ref{cotaoperadordeevolucion}) is only used in order to prove the implication from right to left, that is, it is not necessary in order to prove the other implication.
\end{remark}
\subsection{Main Results}

Now we present the statements of the three main results that we will prove in this article: 
\begin{theorem}
\label{lemafuncionlyapunov}
Consider the couple of system \textnormal{(\ref{lin2})}--\textnormal{(\ref{nolin2})} such that $C_i(t) \in \Sigma(A)$ for $i=1, \ldots, n$ and $ \left \|{B(t)} \right \| \leq \delta K_{\delta, \epsilon}.$ If \textnormal{\textbf{(P1)}--\textbf{(P3)}} are satisfied, then
\begin{itemize}
\item[\textnormal{(1)}] If $y(t)$ is solution of \textnormal{(\ref{nolin2})}, $g\in\mathcal{A}_2$ and $\alpha_1>\mu_1$ (constant defined in Remark \ref{obs1}), then for $L_g <\alpha_1-\mu_1$, we have 
\begin{equation}
\frac{d V(t,y(t))}{dt}\leq -2[\alpha_1-\mu_1-L_g]V(t,y(t)),
\end{equation}
where $V(t,x)$  is a quadratic Lyapunov function associated to \textnormal{(\ref{lin2})}.
\item[\textnormal{(2)}] The systems \textnormal{(\ref{lin2})}--\textnormal{(\ref{nolin2})} are topologically equivalent.
\end{itemize} 
\end{theorem}

\begin{theorem}
\label{previoalprincipal}
If the properties \textnormal{{\textbf{(P1)}}-{\textbf{(P4)}}} are verified with $0 < \delta<\alpha-\mu$ and $f\in\mathcal{A}_2$  such that 
\begin{equation}
L_f\leq \frac{\delta}{M_1^2}, 
\end{equation}
with $\left \| S(\delta,t)\right \|\leq M_1\exp(\beta t)$ and $\left \| S^{-1}(\delta,t)\right \|\leq M_1\exp(\beta t)$ (see Remark 4), then the systems \textnormal{(\ref{lin1})} and \textnormal{(\ref{nolin1})} are topologically equivalent. 
\end{theorem}

\begin{theorem}
\label{ultimo}
If the properties \textnormal{{\textbf{(P1)}}-{\textbf{(P4)}}} are verified with $0 < \delta<\alpha-\mu$ and $f\in\mathcal{A}_1$  such that 
\begin{equation}
\label{cotasLf}
L_f\leq \min \left \{\frac{\delta}{M_1^2}, \frac{\alpha}{K} \right \},
\end{equation}
then the systems\textnormal{(\ref{lin1})} and \textnormal{(\ref{nolin1})} are topologically equivalent.
\end{theorem}

\section{Some Classical Results}

The following proposition is a classical result of local continuity with respect to the initial conditions for differential equations. 

\begin{proposition}
\label{propcongronwall}
Let us consider the differential equation
\begin{equation}
    \label{cualquieraF}
    \dot{x} = F(t,x)
\end{equation}
where $F\in\mathcal{A}_2,$ then for the solution $X(t,s,u)$ of (\ref{cualquieraF}) with $X(s,s,u)=u$, we have that 
$$\left \| u-v \right \| \exp(-L_F|t-s|)\leq \left \| X(t,s,u)-X(t,s,v) \right \|\leq \left \| u-v \right \| \exp(L_F|t-s|).$$

\end{proposition}

\proof{See \cite{Lin2}, Proposition 2.}

\qed 

The following result is an extension to the nonuniform context of \cite[Proposition 5]{Lin2}.

\begin{proposition}
\label{propespaciobanach}
We assume that the system \textnormal{(\ref{lin1})} has a nonuniform exponential dichotomy on $\R_0^{+}$ with constants $K\geq 1$, $\alpha>0$, $\mu\geq0$ and $P(t)=I$ for any $t\in\R_0^{+}$. Let us consider the nonlinear perturbation

\begin{equation}
\label{perturbacionnolineal}
\dot{x}=A(t)x+\mathcal{F}(t,x(t),\kappa),
\end{equation}
where $\mathcal{F}:\R_0^{+}\times\R^{n}\times \boldsymbol{B}\rightarrow \R^{n}$ is a continuous function and $\boldsymbol{B}$ is a Banach space. Besides, $\mathcal{F}$ satisfies the following conditions:

\begin{itemize}
    \item [\textnormal{(i)}] $\mathcal{F}(t,x,\kappa)$ is bounded with respect a $t$, for all $x\in\R^{n}$ and $\kappa\in\boldsymbol{B}$ fixed with the norm 
    $$\left \|\mathcal{F}(x,\kappa)\right \|_{A}=\sup_{t\in\R_0^{+}} \exp(-\mu t)\left \|\mathcal{F}(t,x,\kappa)\right \|.$$

    \item[\textnormal{(ii)}] There exist $L_{\mathcal{F}}>0$ such that 
    %\begin{equation*} 
    $$\left \|\mathcal{F}(t,x_1,\kappa)-\mathcal{F}(t,x_2,\kappa)\right \|\leq L_{\mathcal{F}}\exp(-2\mu t)\left \|x_1-x_2\right \|\quad$$ for any $t\in\R_0^{+}$ and $\kappa\in\boldsymbol{B}$.
    %\end{equation*}
    
    \item[\textnormal{(iii)}]$K_0=\sup_{t\in\R_{0}^{+}, \kappa\in\boldsymbol{B}}\left \|\mathcal{F}(t,0,\kappa) \right \|<+\infty$
\end{itemize}

If $KL_{\mathcal{F}}<\alpha$ then for any fixed $\kappa\in\boldsymbol{B}$ the system \textnormal{(\ref{perturbacionnolineal})} has a unique bounded solution $Z(t,\kappa)$, with the norm $\left \| \cdot \right \|_{A}$, described by
\begin{equation}
\label{solucionacotada}
    Z(t,\kappa)=\int_{0}^{t}\Phi(t,\tau)\mathcal{F}(\tau,Z(\tau,\kappa),\kappa) d\tau.
\end{equation}
such that $\sup_{t\in\R_0^{+}, \kappa\in\boldsymbol{B}} \left \| Z(t,\kappa)\right \|<+\infty$.
\end{proposition}

\proof{
The idea of this result is inspired by the Proposition exhibited in the work of F. Lin in \cite[p. 41]{Lin2}. Let us consider a fixed $\kappa\in\boldsymbol{B}$ and construct the sequence $\left \{\varphi_j\right \}_j$ recursively defined by 
$$\varphi_{j+1}(t,\kappa)=\int_{0}^{t}\Phi(t,\tau)\mathcal{F}(\tau,\varphi_j(\tau,\kappa),\kappa)d\tau$$
and 
$$\varphi_{0}(t,\kappa)=\int_{0}^{t}\Phi(t,\tau)\mathcal{F}(\tau,0,\kappa)d\tau,$$ 
where $\varphi_0(t,\kappa)\in\boldsymbol{C},$ where $\boldsymbol{C}$ is defined by
$$
 \boldsymbol{C}=\left \{ 
 \begin{array}{c} U:\R_0^{+}\times\boldsymbol{B}\rightarrow{\R^{n}} \colon \text{for any} \ \kappa\in\boldsymbol{B}\ \text{fixed,}\\
\left \| U(\kappa) \right \|_{A}<+\infty \
 \text{and}\  U \ \text{is continuous in} \,\, (t, \kappa) 
\end{array}
 \right \}
$$
with $\left \| U(\kappa) \right \|_{A}=\sup_{t\in\R_0^{+}}\exp(-\mu t) \left \| U(t,\kappa)\right \|$.
%$$\boldsymbol{C}=\left \{ U:\R_0^{+}\times\boldsymbol{B}\rightarrow{\R^{n}}/ \left \| U(\kappa) \right \|_{A}=\sup_{t\in\R_0^{+}}\exp(-\mu t) \left \| U(t,\kappa)\right \|<+\infty, \text{for any $\kappa\in\boldsymbol{B}$ fixed}  \right \}.$$

In the first place we will proof that $(\boldsymbol{C},\left \| \cdot \right \|_{A})$ is a Banach space. Indeed, let $\left\{U_n\right \}_{n\in\N}$ be a Cauchy sequence in $\boldsymbol{C}$, then for any $\varepsilon>0$ and for $\tau\in\R_0^{+}$ fixed, there exists $N\in\N$ such that for all $n, m \in\N$ 

$$\left \|U_n(\kappa) - U_m(\kappa)\right \|_{A}=\sup_{t\in\R_0^{+}}\exp(-\mu t) \left \| U_n(t,\kappa)-U_m(t,\kappa)\right \|<\varepsilon,$$
but the expression $\exp(-\mu \tau) \left \| U_n(\tau,\kappa)-U_m(\tau,\kappa)\right \|\leq\left \|U_n - U_m\right \|_{A} $ allows us deduce that 
$$\left \| U_n(\tau,\kappa)-U_m(\tau,\kappa)\right \|\leq \exp(\mu \tau)\varepsilon,$$
therefore $\left \{U_n(\tau,\kappa)\right \}_{n\in\N}$ is a Cauchy sequence in $\R^{n}$, so we obtain a well-defined function $U:\R_0^{+}\times\boldsymbol{B}\rightarrow{\R^{n}}$ which satisfies $U(\tau,\kappa)=\lim_{n\rightarrow{+\infty}}U_n(\tau,\kappa)$ for $\tau, \kappa$ fixed. Also we have
$$\left \| U(\tau,\kappa)-U_n(\tau,\kappa)\right \|=\lim_{m\rightarrow{+\infty}}\left \| U_m(\tau,\kappa)-U_n(\tau,\kappa)\right \|\leq\lim_{m\rightarrow{+\infty}}\exp(\mu\tau)\varepsilon=\exp(\mu\tau)\varepsilon,$$
then from the last inequality we have
$$\exp(-\mu\tau)\left \| U(\tau,\kappa)-U_n(\tau,\kappa)\right \|<\varepsilon,$$
so 
$$\sup_{\tau\in\R_0^{+}}\exp(-\mu\tau)\left \| U(\tau,\kappa)-U_n(\tau,\kappa)\right \|\leq\varepsilon.$$

Thus we can infer that 
$$\left \| U(\kappa) \right \|_{A}\leq\left \|U(\kappa)-U_n(\kappa)\right \|_{A}+\left \|U_n(\kappa)\right \|_{A}<\infty$$
for big $n\in\N$. Additionally $U$ is continuous due to the continuity of $U_k$, then $U\in\boldsymbol{C}$, so $(\boldsymbol{C},\left \|\cdot\right \|_{A})$ is a Banach space.  

Now we will prove by using induction that $\varphi_{j}\in\boldsymbol{C}$ for any $j\in\N\cup\left \{ 0 \right \}$. Indeed, when $\varphi_{j}\in\boldsymbol{C}$, 
%$ \left \| F(\varphi_j,\kappa) \right \|_{A}$ is bounded and 
we estimate $\left\|\varphi_{j+1}(\kappa)\right \|_{A}$:
\begin{equation*} 
\left \|\varphi_{j+1}(t,\kappa)\right \|
%\left \|\int_{0}^{t}\Phi(t,\tau)F(\tau,\varphi_j(\tau,\kappa),\kappa)d\tau\right \| 
\leq\int_{0}^{t}K\exp(-\alpha(t-\tau)+\mu \tau)(L_\mathcal{F}\exp(-2\mu \tau)\left \| \varphi_j(\tau,\kappa)\right \|+K_0),
\end{equation*}
from which it follows that if $K_j=\left \|\varphi_j(\kappa)\right \|_{A}$, then
\begin{equation*}
\begin{array}{rl}
\exp(-\mu t)\left \|\varphi_{j+1}(t,\kappa)\right \|\leq & \displaystyle\int_{0}^{t}K\exp(-\mu (t-\tau)\exp(-\alpha(t-\tau))(L_\mathcal{F}K_j\exp(-\mu\tau)+K_0)d\tau, \\
\\
< &\displaystyle\int_{0}^{t}K\exp(-\alpha(t-\tau))(L_\mathcal{F}K_j+K_0)d\tau, \\
\\
< & \displaystyle\frac{K(L_\mathcal{F}K_j+K_0)}{\alpha}<+\infty
\end{array}
\end{equation*}
and we obtain
$$\left \|\varphi_{j+1}(\kappa)\right \|_{A}=\sup_{t\in\R_0^{+}}\exp(-\mu t)\left \|\varphi_{j+1}(t,\kappa)\right \|\leq\displaystyle\frac{K(K_jL_\mathcal{F}+K_0)}{\alpha}<+\infty .$$

From the above, we can consider a map $T:\boldsymbol{C}\rightarrow\boldsymbol{C}$ given by 
\begin{equation*}
    T( Z(t,\kappa))=\int_{0}^{t}\Phi(t,\tau)\mathcal{F}(\tau,Z(\tau,\kappa),\kappa) d\tau,
\end{equation*}
which is well defined. Since we have that $KL_\mathcal{F}<\alpha$, we affirm that $T
$ is a contraction and the following estimate proves it: 
$$
%\begin{equation*}
\begin{array}{rl}
\left \|  T( Z_1(t,\kappa))- T( Z_2(t,\kappa))\right \| \leq & \displaystyle\int_{0}^{t}KL_\mathcal{F}\exp(-\alpha (t-\tau)+\mu\tau-2\mu\tau)\left \|  Z_1(\tau,\kappa)- Z_2(\tau,\kappa) \right \|d\tau, \\
\\
\left \|T( Z_1(\kappa))- T( Z_2(\kappa)) \right \|_{A} \leq & \displaystyle\frac{KL_\mathcal{F}}{\alpha}\left \|Z_1(\kappa)- Z_2(\kappa) \right \|_{A},
\end{array}
%\end{equation*}
$$
and these computations ensure that $\left \{\varphi_j \right \}$ is the unique sequences in $
\boldsymbol{C}$ satisfying the recursivity stated above.

Now we will prove that $\left \{\varphi_j \right \}$ is a Cauchy sequence in the Banach space $(\boldsymbol{C},\left \|\cdot \right \|_{A})$. We proceed inductively. We observe that, firstly 
\begin{equation*}
\begin{array}{rl}
     \left \|\varphi_1(t,\kappa)-\varphi_0(t,\kappa) \right \| \leq & \displaystyle\int_{0}^{t}K\exp(-\alpha(t-\tau)+\mu \tau) L_\mathcal{F}\exp(-2\mu\tau)\left \|\varphi_0(\tau,\kappa) \right \|d\tau,  \\
     \\
     \leq & \displaystyle KL_\mathcal{F}\frac{KK_0}{\alpha}\int_{0}^{t}\exp(-\alpha(t-\tau))d\tau\leq \bar{K}\frac{KL_\mathcal{F}}{\alpha}, \\
     \\
     \end{array}
     \end{equation*}
    from which we can verify that
    $$ \left \|\varphi_1(\kappa)-\varphi_0(\kappa)\right \|_{A} \leq  \displaystyle\bar{K}\frac{KL_\mathcal{F}}{\alpha} $$
with $\displaystyle\bar{K}=\frac{KK_0}{\alpha}$. As inductive hypothesis, we have that 
$$\displaystyle\left \|\varphi_j(\kappa)-\varphi_{j-1}(\kappa)\right \|_{A} \leq \displaystyle\bar{K}\left (\frac{KL_\mathcal{F}}{\alpha}\right )^j$$  
and therefore
\begin{equation*}
    \begin{array}{rl}
     \left \|\varphi_{j+1}(t,\kappa)-\varphi_j(t,\kappa) \right \| \leq & \displaystyle\int_{0}^{t}K\exp(-\alpha(t-\tau)+\mu \tau) L_\mathcal{F}\exp(-2\mu\tau)\left \|\varphi_j(\tau,\kappa)-\varphi_{j-1}(\tau,\kappa) \right \|d\tau,  \\
     \\
     \leq & \displaystyle \bar{K}KL_F\left (\frac{KL_\mathcal{F}}{\alpha}\right )^j\int_{0}^{t}\exp(-\alpha(t-\tau))d\tau\leq \bar{K}\left (\frac{KL_\mathcal{F}}{\alpha}\right )^{j+1}, \\
     \\
     \left \|\varphi_{j+1}(\kappa)-\varphi_j(\kappa)\right \|_{A} \leq & \displaystyle\bar{K}\left (\frac{KL_\mathcal{F}}{\alpha}\right )^{j+1}. 
\end{array}
\end{equation*}

Moreover for all $\varepsilon>0$ there exists $N(\varepsilon)\in\mathbb{N}$ such that for any $n,m\geq N$ we have that 
\begin{equation*}
\begin{array}{rl}
   \displaystyle \left \|\varphi_n(\kappa)-\varphi_m(\kappa)\right \|_{A}\leq & \bar{K}\left (\frac{KL_\mathcal{F}}{\alpha}\right )^{m+1}\left (1+\frac{KL_\mathcal{F}}{\alpha}+\cdots + \left (\frac{KL_\mathcal{F}}{\alpha}\right )^{n-(m-1)}\right )\\
    \\
     & \displaystyle\leq \bar{K}\left (\frac{KL_\mathcal{F}}{\alpha}\right )^{N} \left (\frac{1-\left (\frac{KL_\mathcal{F}}{\alpha}\right)^{n-m}}1-\frac{KL_\mathcal{F}}{\alpha}\right ),\\
    \\
    & \leq \displaystyle \bar{K}\left (\frac{KL_F}{\alpha}\right )^{N}\left (\frac{1}{1-\frac{KL_F}{\alpha}}\right )<\varepsilon.
\end{array}
\end{equation*}
This inequality proves that $\left \{\varphi_j \right \}$ is a Cauchy sequence in the Banach space $\boldsymbol{C}$ which converges to the fixed point $Z(t,\kappa)$ defined by the equation (\ref{solucionacotada}).

Considering a fixed $\kappa\in\boldsymbol{B}$ we have that $\left \| Z(\kappa)\right \|_{A}<C(\kappa)$. That is, $Z(\cdot,\kappa)\in\boldsymbol{C}$ but its bound $C(\kappa)$ could be determined by $\kappa$. However, we will prove that $C(\kappa)$ has an upper bounded independent of $\kappa$. Indeed, combining the properties (ii), (iii) with the nonuniform exponential dichotomy of the system (\ref{lin1}), we have that
$$
\begin{array}{rc}
\left \|Z(t,\kappa)\right \|\leq\ KL_\mathcal{F} \displaystyle\int_{0}^{t}\exp(-\alpha(t-\tau)+\mu\tau)\exp(-2\mu\tau)\left \|Z(\tau,\kappa)\right \|d\tau & \\
+ KK_0\displaystyle\int_{0}^{t}\exp(-\alpha(t-\tau)+\mu\tau)d\tau,
\end{array}
$$
which implies that
$$
\begin{array}{rc}
\exp(-\mu t)\left \|Z(t,\kappa)\right \|\leq KL_\mathcal{F}\displaystyle\int_{0}^{t}\exp(-\alpha(t-\tau))\exp(-\mu\tau)\left \|Z(\kappa)\right \|_{A}d\tau & \\
+ KK_0\displaystyle\int_{0}^{t}\exp(-\alpha(t-\tau))d\tau.
\end{array} 
$$

Thus, 
$$\exp(-\mu t)\left \|Z(t,\kappa)\right \| \leq\frac{KL_\mathcal{F}C(\kappa)}{\alpha} +\frac{KK_0}{\alpha},$$
and taking supremum over $t\in\R_0^{+}$, we obtain
$$
C(\kappa) \leq \frac{K K_0}{\alpha}\left (1 - \frac{K L_{\mathcal{F}}}{\alpha}\right )^{-1}.
$$
}
\qed

\section{Proof of Theorem \ref{lemafuncionlyapunov}}

We will follow the lines of proof of the Lemma in the article of K.J. Palmer \cite[p. 11]{PalmerEqTop}  in order to obtain properties and conditions to prove the Theorem \ref{lemafuncionlyapunov} and \ref{previoalprincipal}. We point out that in the calculations of the derivative of the quadratic Lyapunov function $V$ with respect $t$ evaluated at the origin, we are considering only the right side derivative.  

Let $x(t)=X(t,\tau,\xi)$ be the solution of (\ref{lin2}) such that $x(\tau)=\xi\neq0$ and\newline $y(t)=Y(t,s,\omega)$ be the solution of (\ref{nolin2}) such that $y(s)=\omega\neq0.$

Since the system \textnormal{(\ref{lin2})} has nonuniform contraction, we can use Proposition \ref{chinoshindawi} (the implication from left to right, see Remark 6) to obtain a symmetric positive definite operator $\mathcal{S}(t)$ which define a quadratic Lyapunov function $V(t)$ associated to the system \textnormal{(\ref{lin2})}. Thus, by using the construction of $V(t)$, the inequality \textnormal{(\ref{cotaparaS})}, Remark \ref{obs1} and the Lipschitz constant $L_g$ of function $g$, we obtain that:

\begin{equation*}
\begin{split}
    \displaystyle\frac{dV(t,y(t))}{dt} =  & \left \langle \mathcal{S}'(t)y(t),y(t)\right \rangle + \left \langle \mathcal{S}(t)[A(t)y(t)+g(t,y(t))],y(t)\right \rangle \\
    & +  \left \langle \mathcal{S}(t)y(t),A(t)y(t)+g(t,y(t))\right \rangle,  \\\\
     = & \left \langle \mathcal{S}'(t)y(t)+\mathcal{S}(t)A(t)y(t)+A^{*}(t)\mathcal{S}(t)y(t),y(t) \right \rangle \\
     & + \left \langle \mathcal{S}(t)g(t,y(t)),y(t)\right \rangle + \left \langle \mathcal{S}(t)y(t),g(t,y(t))\right \rangle, \\\\
     \leq & \left \langle -[Id+2(\alpha_1-\mu_1)\mathcal{S}(t)y(t),y(t)]\right \rangle + 2\left \langle \mathcal{S}(t)g(t,y(t)),y(t)\right \rangle, \\\\
     \leq & \left \langle -2(\alpha_1-\mu_1)\mathcal{S}(t)y(t),y(t)\right \rangle + 2\left \langle \mathcal{S}(t) L_g y(t),y(t)\right \rangle, \\\\
     = & -2(\alpha_1-\mu_1)V(t,y(t))+2 L_g V(t,y(t)), \\\\
     \leq & -2[\alpha_1-\mu_1-L_g]V(t,y(t)),
\end{split}
\end{equation*}
and the first part of the Theorem \ref{lemafuncionlyapunov} follows. Besides, if we consider $x(t)$ as previously defined, then in the last inequality we have
\begin{equation}
\label{estimacionV}
\frac{dV(t,x(t))}{dt}\leq -2[\alpha_1-\mu_1]V(t,x(t))\leq -2[\alpha_1-\mu_1-L_g]V(t,x(t)).
\end{equation}

Now we will prove that second statement of this Theorem. 
From \cite[Lemma 2.4]{Liao} with $\mu(\cdot) = \exp(\cdot)$ and considering $\gamma=\alpha_1-\mu_1-L_g>0$, we have that
\begin{displaymath}
    V(t,x(t))\leq V(s,x(s))\exp(-\bar{\gamma}(t-s)), \quad t\geq s
\end{displaymath}
where $\bar{\gamma}=2\gamma$. From the last inequality we affirm that $V(t,x(t))$ is strictly decreasing and converges to $0$ as $t$ tends to infinity. Now given $\varepsilon'>0$, let $\ell=\ell(\varepsilon')>0$ such that there exists a unique $T=T(\tau,\xi)$ that satisfies 

\begin{displaymath}
    V(T,x(T))=\frac{\ell}{2}.
\end{displaymath}
It is fair to say that $T(\tau,\xi)$ is a continuous function of $(\tau,\xi)$ for $\xi\neq0.$ We use the function $T(\tau,\xi)$ to define a new function as follows: 
 
\begin{equation}
    H(\tau,\xi)= \left\{ \begin{array}{lcc}
             Y(\tau,T(\tau,\xi),X(T(\tau,\xi),\tau,\xi)) &   if  & \xi\neq0, \\
             \\0 &  if & \xi=0. \\
             \end{array}
   \right.
\end{equation} 

Clearly, $H(\tau,\xi)$ is a continuous function when $\xi\neq0$. With the purpose to discuss its continuity at $\xi=0$, we analyze the behaviour of $|T(\tau,\xi)-\tau|$ as $\xi$ tends to 0. By \textbf{(V1)} and Proposition \ref{propcongronwall} we have that
$$
\begin{array}{rcl} 
\displaystyle\frac{\ell}{2}& = & V(T(\tau,\xi),X(T(\tau,\xi),\tau,\xi)), \\\\
& \leq & \mathcal{K}^{2}\exp(2\upsilon T(\tau,\xi)) \left \|X(T(\tau,\xi),\tau,\xi) \right \|^2, \\ \\
& = & \mathcal{K}^{2}\exp(2 \upsilon T(\tau,\xi))\exp(2L_{F}|T(\tau,\xi)-\tau|)\left \| \xi\right \|^{2}.
\end{array}
$$

From the last inequality we can deduce that
\begin{equation}
\label{exponencialtiempo}
\exp(-|T(\tau,\xi)-\tau|)\leq \left (\frac{2 \mathcal{K}\exp(2 \upsilon T(\tau,\xi))\left \|\xi\right \|^2}{\ell}\right )^ { \frac{1}{2L_{F}}},
\end{equation}
where $L_F = |{\bar{a}_1}| + \delta K_{\delta, \varepsilon}.$ From the Remark 1 in \cite{Cast-Huerta} we can infer that the system \textnormal{(\ref{lin2})} has a nonuniform contraction, where its evolution operator admits nonuniformly bounded growth; and therefore this fact allows us deduce that the spectrum satisfies the following condition
$$\Sigma(C(t) + B(t)) = \bigcup_{i=1}^m [\bar{a}_i, \bar{b}_i] \subset (-\infty, 0).$$

Now by \cite[Lemma 2.3]{Liao}, there exists $\eta>0$ such that 
$$\begin{array}{rcl}
     \eta \left \| H(\tau,\xi)\right \|^2 & \leq & V(\tau,H(\tau,\xi)),  \\
     & \leq & V(\tau,Y(\tau,T,X(T,\tau,\xi))),
\end{array}$$
however, we have a string of consequences of algebraic work
$$\left \|\xi\right \|\leq \left (\frac{\ell \exp(-2\upsilon\tau)}{2\mathcal{K}^2}\right )^{\frac{1}{2}}\Rightarrow V(\tau,\xi)\leq \mathcal{K}^2\exp(2\upsilon\tau)\left \|\xi\right \|^2\leq \frac{\ell}{2}\Rightarrow T(\tau,\xi)\leq\tau.$$

It follows by \cite[Lemma 2.4]{Liao} and the estimation (\ref{exponencialtiempo}) that

$$\begin{array}{rcl}
     \left \|H(\tau,\xi)\right \|^2 & \leq & \eta^{-1}\exp(-\bar{\gamma}(\tau-T))V(T,Y(T,T,X(T,\tau,\xi))),  \\\\
     & = & \displaystyle\frac{\eta^{-1}\ell\exp(-\bar{\gamma}(\tau-T))}{2}, \\\\
     & \leq & \displaystyle\left (\frac{\eta^{-1}\ell}{2}\right )\left (\frac{2 \mathcal{K}\exp(2\upsilon T(\tau,\xi))\left \|\xi\right \|^2}{\ell}\right )^ { \frac{\bar{\gamma}}{2L_{F}}}.
\end{array}$$
Moreover, if $\displaystyle\left \|\xi\right \|\leq \left (\frac{\ell \exp(-2\upsilon\tau)}{2\mathcal{K}^2}\right )^{\frac{1}{2}}$ then we obtain that
$$\left \|H(\tau,\xi)\right \|\leq \displaystyle\left (\frac{\eta^{-1}\ell}{2}\right )^{\frac{1}{2}}\left (\frac{2 \mathcal{K}\exp(2\upsilon T(\tau,\xi))\left \|\xi\right \|^2}{\ell}\right )^ {\frac{\bar{\gamma}}{4L_{F}}}.$$

On the other hand, we proceed in the similar form to obtain a converse estimation for $H(\tau,\xi)$. By \cite[Lemma 2.3]{Liao} and Proposition \ref{propcongronwall} we have that

$$
\begin{array}{rcl} 
\displaystyle\frac{\ell}{2}& = & V(T(\tau,\xi),X(T(\tau,\xi),\tau,\xi)), \\ 
& \geq & \eta \left \|X(T(\tau,\xi),\tau,\xi)\right \|^2, \\ 
& \geq & \eta\left \| \xi\right \|^{2}\exp(-2L_{F}|T(\tau,\xi)-\tau|).
\end{array}
$$
and from the last inequality we obtain the following computation:
\begin{equation}
\label{exponencialtiempo2}
\exp(|T(\tau,\xi)-\tau|)\geq \left (\frac{2\eta\left \|\xi\right \|^2}{\ell}\right )^ { \frac{1}{2L_{F}}}.
\end{equation}

Also we can deduce that
$$\left \|\xi\right \|\geq\left (\frac{\ell}{2\eta}\right )^{\frac{1}{2}}\Rightarrow\frac{\ell}{2}\leq\eta\left \|\xi\right \|^2\leq V(\tau,\xi)\Rightarrow T(\tau,\xi)\geq \tau,$$
and when $\left \|\xi\right \|\geq\left (\frac{\ell}{2\eta}\right )^{\frac{1}{2}}$, by the inequality (\ref{exponencialtiempo2}) and the condition \textbf{(V3)}, we obtain that
$$
\begin{array}{rcl}
    \left \|H(\tau,\xi)\right \|^2 & \geq & \displaystyle\frac{\exp(-2\upsilon\tau)}{\mathcal{K}^2}V(\tau,H(\tau,\xi)), \\\\
     & = & \displaystyle\frac{\exp(-2\upsilon\tau)}{\mathcal{K}^2}V(\tau,Y(\tau,T,X(T,\tau,\xi))), \\\\
     & \geq & \displaystyle\frac{\exp(-2\upsilon\tau+\bar{\gamma}(T-\tau))}{\mathcal{K}^2}V(T,Y(T,T,X(T,\tau,\xi))), \\\\
     & = &  \displaystyle\frac{\ell \exp(-2\upsilon\tau)}{2\mathcal{K}^2}\exp(\bar{\gamma}(T-\tau)), \\\\
     & \geq &  \displaystyle\frac{\ell \exp(-2\upsilon\tau)}{2\mathcal{K}^2}\left (\frac{2\eta\left \|\xi\right \|^2}{
     \ell}\right )^ {\frac{\bar{\gamma}}{2L_{F}}}
\end{array}
$$
Rewriting the last estimation:
$$\left \|H(\tau,\xi)\right \|\geq \displaystyle\left (\frac{\ell \exp(-2\upsilon\tau)}{2\mathcal{K}^2}\right )^{\frac{1}{2}}\left (\frac{2\eta\left \|\xi\right \|^2}{\ell }\right )^ {\frac{\bar{\gamma}}{4L_{F}}}.$$

Now we must to prove that if $x(t)$ is a solution of (\ref{lin2}), then $H(t,x(t))$ is a solution of (\ref{nolin2}). 

When $\xi=0$, we can affirm that 
$$H(t,X(t,\tau,\xi))=H(t,0)=0.$$
In the case when $\xi\neq0$, we have that
\begin{displaymath}
\begin{split}
    H(t,X(t,\tau,\xi))= & Y(t,T(t,X(t,\tau,\xi)),X(T(t,X(t,\tau,\xi)),t,X(t,\tau,\xi))),\\
    = & Y(t,T(t,X(t,\tau,\xi)),X(T(t,X(t,\tau,\xi)),\tau,\xi)).
\end{split}
\end{displaymath}
On the one hand we have
\begin{equation}
\label{T(xi)}
\frac{\ell}{2}=H(T(\tau,\xi),X(T(\tau,\xi),\tau,\xi))=H(T(\tau,\xi),X(T(\tau,\xi),t,X(t,\tau,\xi))),
\end{equation}
and on the other hand 
\begin{equation}
\label{T(X(xi))}
\begin{split}
\frac{\ell}{2}= & H(T(t,X(t,\tau,\xi)),X(T(t,X(t,\tau,\xi)),t,X(t,\tau,\xi))),
\end{split}
\end{equation}
and by the equations (\ref{T(xi)}) and (\ref{T(X(xi))}), we deduce that $T(t,X(t,\tau,\xi))=T(\tau,\xi)$. Therefore for all $t, \tau \geq 0$, and $\xi\neq0$, we obtain
\begin{equation}
\label{htiempos}
H(t,X(t,\tau,\xi))=Y(t,T(\tau,\xi),X(T(\tau,\xi),\tau,\xi)),
\end{equation}
which is a solution of the system(\ref{nolin2}). 

Similar to the procedure just shown, we must to define a mapping 
\begin{equation}
    G(\tau,\xi)= \left\{ \begin{array}{lcc}
             X(\tau,S(\tau,\xi),Y(S(\tau,\xi),\tau,\xi)) &   if  & \xi\neq0, \\
             \\0 &  if & \xi=0, \\
             \end{array}
   \right.
\end{equation} 
where $S=S(\tau,\xi)$ is the unique time such that 
\begin{displaymath} 
V(S,y(S))=\frac{\ell }{2}.
\end{displaymath}

We can deducefor the function $G$ similar properties to the function $H$. Moreover we have that when $\xi\neq0$:
\begin{displaymath}
G(t,Y(t,\tau,\xi))=X(t,S(\tau,\xi),Y(S(\tau,\xi),\tau,\xi)),
\end{displaymath} 
which is obtained in a similar manner to (\ref{htiempos}). 

In order to prove that $H(\tau,G(\tau,\xi))=\xi$, we designate $S=S(\tau,y)$ and rewrite $\displaystyle\frac{\ell}{2}$ as follows 
\begin{equation}
\label{L1}
\displaystyle\frac{\ell}{2}=V(T(S,Y(S,\tau,y)),X(T(S,Y(S,\tau,y)),S,Y(S,\tau,y))), 
\end{equation}
 and as it can also be written as follows
 \begin{equation}
 \label{L2}
\displaystyle\frac{\ell}{2}=V(S,Y(S,\tau,y))=V(S,X(S,S,Y(S,\tau,y))),
 \end{equation}
therefore from the equations (\ref{L1}) and (\ref{L2}) we assure that 

\begin{equation}
\label{identidadtiempos}
T(S(\tau,y),Y(S(\tau,y),\tau,y))=S(\tau,y).
\end{equation}

Consequently and in order to prove that $H$ and $G$ are inverse functions of each other we have that 

\begin{equation*}
\begin{split}
    H(\tau,G(\tau,\xi))= & H(\tau,X(\tau,S(\tau,\xi),Y(S(\tau,\xi),\tau,\xi))),\\
    = & Y(\tau,T(S,Y(S,\tau,\xi)),X(T(S,Y(S,\tau,\xi)),S,Y(S,\tau,\xi))),
\end{split}
\end{equation*}
and from the equality(\ref{identidadtiempos}), we obtain the following identity

\begin{equation*}
\begin{split}
    H(\tau,G(\tau,\xi))= & Y(\tau,S,X(S,S,Y(S,\tau,\xi))),\\
    = & Y(\tau,S,Y(S,\tau,\xi))=\xi.
\end{split}
\end{equation*}

Similarly we can ratify that
\begin{equation*}
G(\tau,H(\tau,\xi))=\xi,
\end{equation*}
for all $\tau\in\R_0^{+}$, $\xi\in\R^{n}$.
\qed

\section{Proof Theorem \ref{previoalprincipal}}

This result is a consequence of the Theorem \ref{lemafuncionlyapunov} and it is inspired by the Lemma 2 in the work of F. Lin \cite[p. 47]{Lin2} . Indeed, we have that the systems (\ref{lin1}) and (\ref{lin2}) are topologically equivalent through of the matrix $S(\delta,t)$; moreover the systems (\ref{lin2}) and (\ref{sistematransformado}) are topologically equivalent through of the matrix $S(\delta,t)$ also. If we denote $g(t,y)=S^{-1}(\delta,t)f(t,S(\delta,t)y)$, then $g\in\mathcal{A}_2$ and additionally $L_g=M_1^2L_f$. In fact,  
\begin{displaymath}
\begin{split}
\left \|g(t,y_1)-g(t,y_2)\right \|& =\left \|S^{-1}(\delta,t)f(t,S(\delta,t)y_1)-S^{-1}(\delta,t)f(t,S(\delta,t)y_2)\right \|, \\\\
& \leq M_1 \exp(\beta t)\left \| f(t,S(\delta,t)y_1)-f(t,S(\delta,t)y_2)\right \|, \\\\
& \leq M_1 L_f \exp(\beta t)\exp(-2\beta t)\left \|S(\delta,t)y_1-S(\delta,t)y_2\right \|, \\\\
& \leq M_1^{2}L_f\left \|y_1-y_2\right \|.
\end{split}
\end{displaymath}
Since the conjdition $L_g\leq\delta < \alpha - \mu$ is satisfied, combining the Theorem  \ref{lemafuncionlyapunov} and  the fact that topological equivalence is a equivalence relation, the systems (\ref{lin1}) and (\ref{nolin1}) are topologically equivalent.

\section{Proof Theorem \ref{ultimo}}

This proof follows the lines of the papers \cite{CR2018} and \cite{Lin2}. We take the auxiliar function $f_0(t,x)=f(t,x)-f(t,0)$ and we affirm that when $f\in\mathcal{A}_1$ then $f_0\in\mathcal{A}_2$. Indeed, we have that $f_0(t,0)=0$ and the following estimate: 
$$\left \| f_0(t,x_1)-f_0(t,x_2)\right \|=\left \| f(t,x_1)-f(t,x_2)\right \|\leq L_f \exp(-2\beta t) \left \|x_1-x_2\right \|),$$
for any $t\in\R_0^{+}, x_1, x_2\in\R^{n}$ and some $\beta\geq0$. Since $f$ and $f_0$ have the same Lipschitz constant, then by Theorem \ref{previoalprincipal} and inequality (\ref{cotasLf}) it is enough to prove that  the systems (\ref{nolin1}) and 
\begin{equation}
\label{origen0}
\dot{x}=A(t)x+f_0(t,x)
\end{equation}
are topologically equivalent. 
The condition {\textbf{(P3)}} ensures the existence of constants $K\geq 1$, $\alpha >0$ and $\mu\geq0$ satisfying the inequality (\ref{estabilidadexponencialnouniforme}). Besides for the unique solution $X(t,\tau,\xi)$ of system (\ref{origen0}) passing through $\xi$ at $t=\tau$, we must to define the function $F:\R_0^{+}\times\R^n\times\boldsymbol{B}\rightarrow \R^n$, with $\boldsymbol{B}=\R_0^{+}\times\R^n$, as follows

\begin{displaymath}
\begin{split}
F(t,y,(\tau,\xi))& =f(t,y+X(t,\tau,\xi))-f_0(t,X(t,\tau,\xi)), \\
& = f(t,y+X(t,\tau,\xi))-f(t,X(t,\tau,\xi))+f(t,0), \\ 
%& \leq M_1 \eta \exp(\beta t)\exp(-2\beta %t)\left \|H^{-1}(t)y_1-H^{-1}(t)y_2\right %\|, \\
%& \leq M_1^{2}\eta\left \|y_1-y_2\right %\|.
\end{split}
\end{displaymath}
moreover if we define $K_0=\sup\left \{t\in\R_0^{+}\ \colon \ \left \|f(t,0)\right \|\right \}$ then we deduce the following estimates:
\begin{displaymath}
\left\{ \begin{array}{lcc}
             \left \| F(t,y,(\tau,\xi))\right \|\leq L_f\exp(-2\beta t)\left \|y\right \|+K_0, \\
             \\ \left \| F(t,y_1,(\tau,\xi))-F(t,y_2,(\tau,\xi))\right \|\leq L_f\exp(-2\beta t)\left \| y_1-y_2\right \|. 
             \end{array}
   \right.
%\left \| F(t,y,(\tau,\xi))\right \|\leq L_f\exp(-2\beta t)\left \|y\right \|+K_0,
%\left \| F(t,y_1,(\tau,\xi))-F(t,y_2,(\tau,\xi))\right \|\leq\eta\exp(-2\beta t)\left \| y_1-y_2\right \|.
\end{displaymath}

We observe that the two last inequalities allow us to affirm that $F$ verifies the hypothesis of Proposition \ref{propespaciobanach}, which implies that the system 
\begin{equation}
\label{sistema4variables}
\dot{z}=A(t)z+F(t,z,(\tau,\xi))
\end{equation}
has an unique bounded and continues solution $Z(t,(\tau,\xi))$ defined by 
\begin{displaymath}
Z(t,(\tau,\xi))=\displaystyle\int_{0}^{t}\Phi(t,s)[f(s,Z(s,(\tau,\xi))+X(s,\tau,\xi))-f_0(s,X(s,\tau,\xi))]ds
\end{displaymath}
and this solution is bounded and continuous with the norm 
$$\left \|Z\right \|=\sup_{t\in\R_0^{+},(\tau,\xi)\in\R_0^{+}\times\R^n}\exp(-\beta t)\left \| Z(t,(\tau,\xi))\right \|=M_0<+\infty.$$

Now, let us construct the map $H:\R_0^{+}\times\R^{n}\rightarrow \R^{n}$ as follows
\begin{equation}
\label{homeo}
    H(\tau,\xi)=\xi+Z(\tau,(\tau,\xi)).
\end{equation}

\begin{lemma}
For any $(r,t)\in\R_0^{+}\times\R_0^{+}$ and $(\tau,\xi)\in\R_0^{+}\times\R^n$, the following identity is satisfied: 
\begin{equation}
\label{solucionacotadacondicion}
    Z(r,(t,X(t,\tau,\xi)))=Z(r,(\tau,\xi)).
\end{equation}
\end{lemma}

\begin{proof}
Firstly, we describe the solution of the system \ref{sistema4variables} as follows
%\begin{array}{rc}
\begin{equation*}
\begin{split}
Z(r,(t,X(t,\tau,\xi)) & = \displaystyle\int_{0}^{r}\Phi(r,s)[f(s,Z(s,(t,X(t,\tau,\xi)))+X(s,t,X(t,\tau,\xi))) \\
& \quad -f_0(s,X(s,t,X(t,\tau,\xi)))] ds,\\
& = \displaystyle\int_{0}^{r}\Phi(r,s)[f(s,Z(s,(t,X(t,\tau,\xi)))+X(s,\tau,\xi))] \\
& \quad -f_0(s,X(s,\tau,\xi))] ds
\end{split}
%\end{array}
\end{equation*}
and also we have 
\begin{equation*}
    Z(r,(\tau,\xi))=\displaystyle\int_{0}^{r}\Phi(r,s)[f(s,Z(s,(\tau,\xi))+X(s,\tau,\xi))-f_0(s,X(t,\tau,\xi))] ds.
\end{equation*}

Secondly, we estimate the norm of the subtraction between $Z(r,(t,X(t,\tau,\xi)))$ and $Z(r,(\tau,\xi))$:

\begin{equation*}
\begin{array}{lc}
    \left \| Z(r,(t,X(t,\tau,\xi)))-Z(r,(\tau,\xi))\right \|=\\ 
    \\
     \left \| \displaystyle\int_{0}^{r}\Phi(r,s)[f(s,Z(s,(t,X(t,\tau,\xi)))+X(s,\tau,\xi))-f(s,Z(s,(\tau,\xi))+X(s,\tau,\xi))]ds\right \|\\
     \\
     \leq \displaystyle\int_{0}^{r}KL_f\exp(-\alpha(r-s)+\mu s)\exp(-2\mu s)\left \|Z(s,(t,X(t,\tau,\xi)))-Z(s,(\tau,\xi)) \right \| ds, \\
     \\
     \leq \displaystyle\frac{KL_f}{\alpha}\sup_{r\in\R_0^{+}}\exp(-\mu r)\left \|Z(r,(t,X(t,\tau,\xi)))-Z(r,(\tau,\xi)) \right \| 
\end{array}
\end{equation*}
and since $\frac{KL_f}{\alpha}<1$, the the Lemma is fulfilled.
\end{proof}

\begin{lemma}
If $t\mapsto X(t,\tau,\xi)$ is solution of the system  \textnormal{(\ref{origen0})} such that $X(\tau,\tau,\xi)=\xi$, then $t\mapsto H(t,X(t,\tau,\xi))$ is solution of the system \textnormal{(\ref{nolin1})}.
\end{lemma}

\begin{proof}
Combining the equations (\ref{homeo}) and (\ref{solucionacotadacondicion}), we have that
\begin{equation*}
    H(t,X(t,\tau,\xi))=X(t,\tau,\xi)+Z(t,(\tau,\xi)), 
\end{equation*}
and a simple computation allows us to verify the statement.
\end{proof}

\begin{lemma}
The map $\xi\mapsto H(\tau,\xi)$ is a continuous function for any fixed $\tau\in\R_0^{+}$.
\end{lemma}

\begin{proof}
By the equation (\ref{homeo}), the only thing that we should to prove is that the map $\xi\mapsto Z(\tau,(\tau_0,\xi))$ is a continuous function for any fixed $\tau$. Indeed, let us recall that $\tau\mapsto Z(\tau,(\tau,\xi))$ is the unique bounded solution in $\boldsymbol{C}$ of the system (\ref{sistema4variables}), which was constructed by successive approximations in Proposition \ref{propespaciobanach}, this means that we have
\begin{displaymath}
\label{limiteZ}
    \lim_{j\rightarrow +\infty}Z_j(\tau,(\tau_0,\xi))=Z(\tau,(\tau_0,\xi)),
\end{displaymath}
where 
\begin{equation*}
    Z_{j+1}(\tau,(\tau_0,\xi))=\displaystyle\int_{0}^{\tau}\Phi(\tau,s)F(s,Z_j(s,(\tau_0,\xi)),(\tau_0,\xi)) ds.
\end{equation*}

Additionally we know that for any $\varepsilon>0$, there exists $J=J(\varepsilon)>0$ such that for any $j>J$ it follows that 
\begin{equation*}
\begin{split}
\left \|Z(\tau,(\tau_0,\xi))-Z(\tau,(\tau_0,\xi')) \right \| & \leq \left \| Z(\tau,(\tau_0,\xi))-Z_j(\tau,(\tau_0,\xi))\right \| \\\\
& \quad +\left \| Z_j(\tau,(\tau_0,\xi))-Z_j(\tau,(\tau_0,\xi'))\right \|\\\\
& \quad + \left \| Z_j(\tau,(\tau_0,\xi'))-Z(\tau,(\tau_0,\xi'))\right \| \\
& < \frac{2}{3}\varepsilon + \left \| Z_j(\tau,(\tau_0,\xi))-Z_j(\tau,(\tau_0,\xi'))\right \|.
\end{split}
%\end{array}
\end{equation*}

We will prove by induction that for any $j\in\mathbb{N}$, there exists $\delta_j>0$ such that
\begin{equation}
\label{Zjcontinuidad}
    \left \| Z_j(\tau,(\tau_0,\xi))-Z_j(\tau,(\tau_0,\xi'))\right \|<\frac{\varepsilon}{3} \quad \textnormal{when}\quad \left \| \xi-\xi'\right \|<\delta_j.
\end{equation}

Indeed, we cosider an initial term
$$Z_0(\tau,(\tau,\xi))=Z_0(\tau,(\tau,\xi'))=\phi_0\in\boldsymbol{C}$$
and we suppose that (\ref{Zjcontinuidad}) is verified for some $j$ as inductive hypothesis. Now, we have that
$$\left \| Z_{j+1}(\tau,(\tau_0,\xi))-Z_{j+1}(\tau,(\tau_0,\xi'))\right \|\leq \Delta$$
where
\begin{equation*}
    \Delta= \displaystyle\left \| \int_{0}^{\tau}\Phi(\tau,s)[F(s,Z_j(s,(\tau_0,\xi)),(\tau_0,\xi))-F(s,Z_j(s,(\tau_0,\xi')),(\tau_0,\xi'))] ds\right \|.
\end{equation*}
%and 
%\begin{equation*}
 %   \Delta_2= \displaystyle\left \| \int_{\tau-R}^{\tau}\Phi(\tau,s)[F(s,Z_j(s,(\tau_0,\xi)),(\tau_0,\xi))-F(s,Z_j(s,(\tau_0,\xi')),(\tau_0,\xi'))] ds\right \|.
%\end{equation*}

From the definition and properties of $F$, by Gronwall's Lemma and inductive hypothesis, we have that
\begin{equation*}
\begin{split}
    \Delta &\leq\displaystyle\left \| \int_{0}^{\tau}\Phi(\tau,s)[f(s,Z_j(s,(\tau_0,\xi))+X(s,\tau_0,\xi))-f(s,Z_j(s,(\tau_0,\xi'))+X(s,\tau_0,\xi'))] ds\right \| \\
    \\ 
    & \quad + \displaystyle\left \| \int_{0}^{\tau}\Phi(\tau,s)[f_0(s,X(s,\tau_0,\xi))-f_0(s,X(s,\tau_0,\xi'))] ds\right \|,\\
    \\
    & \leq \int_{0}^{\tau}KL_f\exp(-\alpha(\tau-s)-\mu s)[\left \|Z_j(s,(\tau_0,\xi))-Z_j(s,(\tau_0,\xi'))\right \| + \left \|X(s,\tau_0,\xi)-X(s,\tau_0,\xi') \right \|] \\ 
    \\ 
    & \quad + \int_{0}^{\tau}KL_{f_0}\exp(-\alpha(\tau-s)-\mu s)[\left \|X(s,\tau_0,\xi)-X(s,\tau_0,\xi') \right \|], \\
    \\ 
    & \leq \displaystyle\frac{\varepsilon}{3}KL_f\int_{0}^{\tau}\exp(-\alpha(\tau-s)-\mu s+\mu s)ds \\
    \\
    & \quad + K(2L_f)\int_{0}^{\tau}\exp(-\alpha(\tau-s))\left \|\xi-\xi'\right \|\exp(L_F(\tau-s))ds,\\
    \\
    & \leq \displaystyle\frac{\varepsilon}{3}\frac{KL_f}{\alpha}+\frac{K(2L_f)}{\alpha}\exp(L_f\tau)\left \|\xi-\xi'\right \|
\end{split}
\end{equation*}
and (\ref{Zjcontinuidad}) is satisfied for $j+1$ when we choose $$\displaystyle\delta_{j+1}=\min\left \{\delta_j,\left (1-\frac{KL_f}{\alpha}\right )\exp(-L_f\tau)\frac{\alpha}{K(2L_f)}\frac{\varepsilon}{3}\right \}$$
and the continuity of $\xi\mapsto Z(\tau,(\tau_0,\xi))$ is satified. All of the above allows us to conclude that $H$ is continuous for any fixed $\tau$.

\end{proof}

\begin{remark}
We note that if $Y(t,\tau,\xi)$ is the unique solution of the system \textnormal{(\ref{nolin1})} passing through $\xi$ at $t=\tau$, we must to define the function $\tilde{F}:\R_0^{+}\times\R^n\times\boldsymbol{B}\rightarrow\R^n$ as follows 
\begin{displaymath}
\begin{split}
    \tilde{F}(t,\tilde{y},(\tau,\xi)) &= f_0(t,\tilde{y}+Y(t,\tau,\xi))-f(t,Y(t,\tau,\xi)), \\
    &= f(t,\tilde{y}+Y(t,\tau,\xi))-f(t,0)-f(t,Y(t,\tau,\xi)).
\end{split}
\end{displaymath}
and we obtain the following estimates: 
\begin{displaymath}
\left\{ \begin{array}{lcc}
             \left \| \tilde{F}(t,\tilde{y},(\tau,\xi))\right \|\leq L_f\exp(-2\beta t)\left \|\tilde{y}\right \|+K_0, \\
             \\ \left \| \tilde{F}(t,\tilde{y}_1,(\tau,\xi))-\tilde{F}(t,\tilde{y}_2,(\tau,\xi))\right \|\leq L_f\exp(-2\beta t)\left \| \tilde{y}_1-\tilde{y}_2\right \|. 
             \end{array}
   \right.
\end{displaymath}

In the same form $\tilde{F}$ satisfies the hypothesis of Proposition \ref{propespaciobanach}, which implies that the system

\begin{displaymath}
    \dot{z}=A(t)z+\tilde{F}(t,z,(\tau,\xi))
\end{displaymath}
has an unique bounded and continuous solution $\tilde{Z}(t,(\tau,\xi))$ defined by 

\begin{displaymath}
    \tilde{Z}(t,(\tau,\xi))=\displaystyle\int_{0}^{t}\Phi(t,s)[f_0(s,\tilde{Z}(s,(\tau,\xi))+Y(s,\tau,\xi))-f(s,Y(s,\tau,\xi))]ds.
\end{displaymath}
\end{remark}

As a consequence of the previous Remark, we construct the map \newline $G:\R_0^{+}\times\R^n\rightarrow\R^n$ as follows
\begin{displaymath}
    G(\tau,\xi)=\xi+\tilde{Z}(\tau,(\tau,\xi)).
\end{displaymath}
and we prove the following results that are similar to the previous one.

\begin{lemma}
For any $(r,t)\in\R_0^{+}\times\R_0^{+}$ and $(\tau,\xi)\in\R_0^{+}\times\R^n$ the following identify is satified:
\begin{equation*}
    \tilde{Z}(r,(t,Y(t,\tau,\xi)))=\tilde{Z}(r,(\tau,\xi)).
\end{equation*}
\end{lemma}
\begin{lemma}
If $t\mapsto Y(t,\tau,\xi)$ is solution of the system \textnormal{(\ref{lin1})} such that $Y(\tau,\tau,\xi)=\xi$, then $t\mapsto G(t,Y(t,\tau,\xi))$ is solution of the system \textnormal{(\ref{origen0})}.
\end{lemma}
\begin{lemma}
The map $\xi\mapsto G(\tau,\xi)$ is a continuous function for any fixed $\tau\in\R_0^{+}$.
\end{lemma}

Finally, from all these Lemmas, we can conclude that the systems (\ref{nolin1}) and (\ref{origen0}) are topologically equivalent, which is enough to prove the result.

 %Due to Sacker and Sell spectrum is the contractible set of the system \textnormal{(\ref{lin1})} when this system has uniform exponential dichotomy (\textit{i.e} satisfies Definition \ref{NUED} with $\mu = 0$), F. Lin in \cite{Lin2} proved that 

\end{document}